\documentclass{amsart}
\usepackage{amssymb,latexsym,amsmath,amsthm}
\usepackage{hyperref}
\usepackage[mathscr]{eucal}

\newtheorem{theorem}{Theorem}[section]
\newtheorem{lemma}[theorem]{Lemma}

\theoremstyle{definition}

\newtheorem{definition}[theorem]{Definition}
\numberwithin{equation}{section}

\author[Shuxin Wang]{Shuxin Wang}
\address{Department of Mathematics and Statistics, University of New Mexico,
  Albuquerque, NM 87131, USA}
\email{wsxmath@unm.edu}

\begin{document}

   \title{Well-posedness and Ill-posedness for the Nonlinear Beam Equation }

   \begin{abstract}

     We investigate Strichartz estimates for the nonlinear beam equation with initial data $f\in\dot{H}^s, g\in\dot{H}^{s-2}$ and $f\in H^s, g\in H^{s-2}$. We extend results of H. Lindblad and C. D.Sogge [10] and T. Cazenave and F. B. Weissler [4] to nonlinear beam equations to determine the minimal regularity that is needed to prove well-posedness and scattering results with low regularity data. Finally, we also use small dispersion analysis of  M. Christ,  J. Colliander and T. Tao [2] to prove the nonlinear beam equation is ill-posed in defocusing case $\omega=-1$ when $ 0<s<s_c=\frac{n}{2}-\frac{4}{\kappa-1}$.
  \end{abstract}

   \maketitle
\section{Introduction}
In recent years, various models involving the beam equations have been studied. Peletier and Troy [13] presented several such nonlinear equation models in the physics literature. E.Cordero and D.Zucco [1] studied dispersive properties of the linear beam equation.  B. Pausader [11], [12] investigated the well-posedness and scattering theory for nonlinear beam equations in the energy space.
In this paper, we consider the Cauchy problem for the nonlinear  beam equation 
\begin{equation}
\left\{ \begin{array}{lcl}
\partial_{t}^{2}u(t,x)+\triangle^{2}u(t,x)=\omega |u|^{\kappa-1}u(t,x),\\
u\mid_{t=0}=f(x) \\
\partial _{t}u\mid_{t=0}=g(x),\end{array}\right.
\end{equation}
where, $\omega=\pm 1$ and $1<\kappa<\infty,$ and $u:\mathbb{R}\times \mathbb{R}^n\rightarrow \mathbb{C}.$
The equation (1.1)  is said to be $defocusing$ when $ \omega<0$, and $focusing$ when $\omega>0$. We investigate the global and local well-posedness in fractional homogeneous and inhomogeneous Sobolev spaces for the Cauchy problem of this equation under minimal regularity assumptions on the initial data. The works [9], [17] used that the beam equation (1.1) with $\omega=0$ can be factorized as the following product $$(\partial_{t}^{2}+\triangle^{2})u=(i\partial_{t}+\triangle)(-i\partial_{t}+\triangle)u,$$ which displays the relation with the Schr\"{o}dinger equation. This suggests one can recover Strichartz estimates for the beam equation from  the ones for the Schr\"{o}dinger equation. Some classical references on Strichartz estimates for the Schr\"{o}dinger equation are provided by [3],[5],[8],[15].  However, the beam equation doesn't satisfy finite speed of propagation, and this turns out to be a source of difficulities in obtaining results analogous to the wave equation.

\begin{definition}
The $ inhomogeneous$ $Sobolev$ $  space$ $ W^{s,r}$ and  the $ homogeneous$ $Sobolev$ $space$ $\dot{W}^{s,r}$ are defined for  $s\in\mathbb{R}$ and $ 1<r< \infty$ as the closure of Schwartz functions $f$ under their respective norms
$$\|f\|_{W^{s,r}}=\|\langle D \rangle^s f\|_{L^r},$$
$$\|f\|_{\dot{W}^{s,r}}=\||D|^s f\|_{L^r},$$
where the fractional differentiation operators $\langle D \rangle^s$ and $|D|^s$ are the Fourier multipliers defined by
$$ \widehat{\langle D\rangle ^s f}(\xi):=\langle \xi\rangle ^s\hat{f}(\xi) \quad \text{and} \quad \widehat{| D| ^s f}(\xi):=| \xi|^s\hat{f}(\xi).$$
In particular, if $s=2$ then $\langle D\rangle^s= I-\triangle,$ where $I$ is the identity operator, and $|D|^s =-\triangle$. If $r=2$ these spaces are also denoted by $H^s(\mathbb{R}^n)$ and $\dot{H}^s(\mathbb{R}^n)$.
\end{definition}
The nonlinear beam equation (1.1) enjoys the scaling symmetry 
\begin{equation}
u(t,x)\mapsto\lambda ^{\frac{-4}{\kappa-1}}u\left (\frac{t}{\lambda^2},\frac{x}{\lambda}\right );\\\quad
f(x)\mapsto\lambda ^{\frac{-4}{\kappa-1}}f\left (\frac{x}{\lambda}\right );\\\quad
g(x)\mapsto\lambda ^{\frac{-4}{\kappa-1}-2}g\left (\frac{x}{\lambda}\right ).
\end{equation} If we compute the initial data  $\|\lambda ^{\frac{-4}{\kappa-1}}f\left (\frac{x}{\lambda}\right )\|_{\dot{H}^s}$ we see that 
\begin{equation}
\|\lambda ^{\frac{-4}{\kappa-1}}f\left (\frac{x}{\lambda}\right )\|_{\dot{H}^s}\thicksim \lambda^{-s+s_c}\|f\|_{\dot{H}^s},
\end{equation}
where$$s_c:=\frac{n}{2}-\frac{4}{\kappa-1}$$
 is the $critical$ $regularity$. This scale invariance predicts a relationship between time existence and regularity of initial data (see Principle 3.1 of Tao [16]). We expect that (1.1) is ill-posedness when $s<s_c$ and well-posedness when $s\geq s_c$.

  In this paper, we start out with the introduction of mixed space-time integrability estimates known as  Strichartz estimates for the beam equation. We extend the results of  E.Cordero, D.Zucco [1] for the linear beam equation and B. Pausader [11] for nonlinear beam equation in the energy space. Here we introduce further studies with initial data $f\in\dot{H}^s, g\in\dot{H}^{s-2}$ and also give these estimates in  inhomogeneous Sobolev spaces ${H}^s\times{H}^{s-2}$.   

   The local and global well-posedness of semilinear dispersive equations has attracted a lot of attention in the past years. In general, when global well-posedness is established, the existence of a scattering operator, comparing the nonlinear dynamics and the linear one, is a rather direct by-product.  H. Lindblad and C. D.Sogge [10] and T. Cazenave and F. B. Weissler [4] proved existence and for semilinear wave and Schr\"{o}dinger equations with low regularity data and determine the minimal Sobolev regularity that is needed to ensure local well-posedness.  H. Lindblad and C. D.Sogge [10] took advantage of the Strichartz estimates to prove a well-posedness theorem for the nonlinear wave equation with rough initial data by the Picard iteration method. By this method,  in section 3, we investigate well-posedness with initial data  $f(x)\in\dot{ H}^{s}(\mathbb{R}^n),g(x)\in \dot{ H}^{s-2}(\mathbb{R}^n)$, and $f(x)\in H^{s}(\mathbb{R}^n),g(x)\in H^{s-2}(\mathbb{R}^n)$ for ``energy critical", ``energy subcritical" exponents $\kappa \leq \frac{n+4}{n-4}$ and ``energy supercritical" exponents $\kappa > \frac{n+4}{n-4},$ and  determine the minimal Sobolev regularity that is needed to ensure local and global well-posedness for the nonlinear beam equation.
Since the beam equation doesn't satisfy finite speed propagation, we use a fractional chain rule to deal with the ``energy super critical" case. 

 Section 4 is concerned the asymptotic completeness and scattering for small amplitude solutions.
  
 There are certain equations and certain regularities for which the Cauchy problem is ill-posed. M. Christ,  J. Colliander and T. Tao [2] give examples of solution to nonlinear wave and Schr\"{o}dinger equations on $\mathbb{R}^n$ which show that problem is ill-posed in the Sobolev space when the exponent $s$ is lower than the critical exponent predicted by scaling. Then in the last section we discuss the ill-posedness results for the Cauchy problem of the nonlinear beam equation with $0<s<s_c$  by small dispersion analysis of  M. Christ,  J. Colliander and T. Tao.

\section*{Acknowledgements}
This work is part result of the author’s doctoral dissertation research at the University of New Mexico. The author heartily acknowledged Professor Matthew Blair for his introduction to the problem, his continuing to encourage and guide throughout the research.

 \section{Strichartz estimates }
\label{main}
We first introduce some notations and definitions that will be frequently used in this paper. The expression $X \lesssim Y$ means $X\leq CY$ for some constant $C$. The mixed Strichartz space-time norm is defined as the following 
$$\|u\|_{L^p_I L^r(\mathbb{R}^n)}=\left [\int_I\left (\int_{\mathbb{R}^n}|u(t,x)|^r dx\right )^{\frac{p}{r}}dt\right ]^{\frac{1}{p}}.$$
The Strichartz estimates involve the following definitions:
\begin {definition}
We say that the exponent pair $(p,q)$ is a Schr$\text{\"{o}}$dinger-admissible pair if
$$2\leq p,q\leq \infty,\quad  \frac{2}{p}+\frac{n}{q}=\frac{n}{2}, \quad n\geq 1, \quad (p,q,n)\neq(2,\infty,2).$$
\end{definition}

\begin{definition}
We say that the exponent triple $(p,r,s)$ is a beam-admissible triple if $s\geq 0$ and 
$$2\leq p,r\leq \infty,\quad  \frac{2}{p}+\frac{n}{r}=\frac{n}{2}-s,\quad  n\geq 2, \quad (p,r,n)\neq(2,\infty,2).$$
\end{definition}
Consider the linear  beam equation, %$u(t,\cdot)\in (\mathbb{R}\times{ \mathbb{R}}^n)$ 
\begin{equation}
\left\{ \begin{array}{lcl}
\partial_{t}^{2}u +\triangle^{2}u=F,\\
u\mid_{t=0}=f,\\
\partial _{t}u\mid_{t=0}=g.\end{array}\right.
\end{equation}
The solution of this equation can be formally written in the integral form 
$$ u(t,\cdot)=\cos(t\triangle)f +\frac{\sin(t\triangle)}{\triangle}g+\int_{0}^{t}\frac{\sin((t-s)\triangle)}{\triangle}F(s)ds.$$
We have the following theorem about Strichartz estimates for solutions to the  beam equation with initial data $f\in\dot{H}^s, g\in\dot{H}^{s-2}$.
\begin{theorem}
Let $n\geq1, s\in\mathbb{R}$, $I$ be either the interval $[ 0, T]$, $T>0$, or $[0, \infty)$, $(p,r,s)$ be a beam-admissible triple, (a,b) is a Schr$\ddot{o}$dinger-admissible pair, and $ (a',b')$ is the conjugate pair of $(a,b)$. If u is a solution to the Cauchy problem  (2.1), then we have the following estimates:
\begin{multline}
\|u\|_{L^{p}_IL^r}+\|u\|_{L^\infty_I{\dot{H}}^s(\mathbb{R}^n)}+\|\partial_t u\|_{L^\infty_I{\dot{H}}^{s-2}(\mathbb{R}^n)}\\
\lesssim\|f\|_{\dot{H}^s}+\|g\|_{\dot{H}^{s-2}}+\|F\| _{L^{a'}_I\dot{W}^{s-2,b'}} ,
\end{multline}
with implicit constant independent of  $I$.
In particular, when $ 0\leq s\leq 2$, $\tilde{b}=\frac{nb'}{n+(2-s)b'},  \tilde{b}>1. $
\begin{equation}
\|u\|_{L^{p}_IL^r}+\|u\|_{L^\infty_I{\dot{H}}^s(\mathbb{R}^n)}+\|\partial_t u\|_{L^\infty_I{\dot{H}}^{s-2}(\mathbb{R}^n)}\lesssim\|f\|_{\dot{H}^s}+\|g\|_{\dot{H}^{s-2}}+\|F\| _{L^{a'}_IL^{\tilde{b}}},
\end{equation}
with implicit constant independent of  $T$,  where ,
$$\frac{2}{p}+\frac{n}{r}=\frac{n}{2}-s=\frac{2}{a'}+\frac{n}{\tilde{b}}-4.$$

\end{theorem}

\begin{proof}
By the work of E.Cordero, D.Zucco [1], the following estimates hold
\begin{equation}
\|u\|_{L^{p}_I\dot{W}^{s,q}}\lesssim\|f\|_{\dot{H}^s}+\|g\|_{\dot{H}^{s-2}}+\|F\| _{L^{a'}_I\dot{W}^{s-2,b'}}.
\end{equation} 
Where $(p,q) $ and $ (a,b)$ are Schr$\text{\"{o}}$dinger-admissible pairs.
For fixed $t$, by Sobolev embedding, we have
$$\|u(t,\cdot)\|_{L^r}\lesssim\|u(t,\cdot)\|_{\dot{W}^{s,q}},$$
when $\frac{1}{q}=\frac{1}{r}+\frac{s}{n},$
combining with 
$ \frac{2}{p}+\frac{n}{q}=\frac{n}{2}$
we have 
$$\frac{2}{p}+\frac{n}{r}=\frac{n}{2}-s.$$
Therefore we have the estimate
\begin{equation}
\|u\|_{L^{p}_IL^r}\lesssim\|f\|_{\dot{H}^s}+\|g\|_{\dot{H}^{s-2}}+\|F\| _{L^{a'}_I\dot{W}^{s-2,b'}} .
\end{equation}

Let $v$ be the solution of (2.1) with $F(t,x)=0$, $ w$ be the solution of (2.1) with vanishing initial data. Then the solution of (2.1) is  $u=v+w$. By the energy inequality  for the linear Cauchy problem, we have
\begin{equation}
\|v\|_{L^\infty_I{\dot{H}}^s(\mathbb{R}^n)}+\|\partial_t v\|_{L^\infty_I{\dot{H}}^{s-2}(\mathbb{R}^n)}\leq 2(\|f\|_{\dot{H}^s}+\|g\|_{\dot{H}^{s-2}}).
\end{equation}
For $w(t,\cdot)=\int_{0}^{t}\frac{\sin((t-s)\triangle)}{\triangle}F(s)ds$, in the non-endpoint case, by dualities of the operators $e^{it\triangle}$ and $\frac{e^{it\triangle}}{\triangle}$, the Christ-Kiselev lemma, and the endpoint Strichartz estimates of Keel-Tao [8], we have the inhomogeneous Strichartz estimates
$$\left\|\int_{0}^{t}\frac{e^{i(t-s)\triangle}}{\triangle}F(s)ds\right\|_{L^p_I\dot{W}^{s,q}}\lesssim||F\|_{L^{a'}_I\dot{W}^{s-2,b'}},$$
where $(p,q), (a,b)$ are Schr$\text{\"{o}}$dinger-admissible pairs. When $p=\infty$, $q=2$, by the definition of the homogeneous Sobolev space,
we have 
$$\|(-\triangle)^{\frac{s}{2}}w\|_{L^\infty_I{\dot{H}}^s(\mathbb{R}^n)}\lesssim\|F\|_{L^{a'}_I\dot{W}^{s-2,b'}}.$$
Then we have,
\begin{equation}
\|w\|_{L^\infty_I{\dot{H}}^s(\mathbb{R}^n)}+\|\partial_t w\|_{L^\infty_I{\dot{H}}^{s-2}(\mathbb{R}^n)}\lesssim\|F\|_{L^{a'}_I\dot{W}^{s-2,b'}}.
\end{equation}
Combining with (2.5),(2.6) and (2.7), we have the estimates (2.2).

Since,
$$\|F\|_{L^{a'}_I \dot{W}^{s-2,b'}}=\|(-\triangle) ^{\frac{s-2}{2}}F\|_{L^{a'}_I L^{b'}}.$$
Now assume $s\leq2$ , for fixed $t$, by the Theorem 1 of chapter 5 in [14] (which is equivalent to Sobolev embedding),$$\|(-\triangle)^{\frac{s-2}{2}}F(t,\cdot)\|_{L^{b'}}\lesssim\|F(t,\cdot)\|_{L^{\tilde{b}}},$$
where,
$\frac{1}{\tilde{b}}=\frac{1}{b'}+\frac{2-s}{n}.$
Then by the same way of proving (2.2), we have the Strichartz estimates (2.3)

\end{proof}

Now we consider the Strichartz estimates for solutions to the  beam equation with initial data $f\in H^s, g\in H^{s-2}$ (inhomogeneous Sobolev space), we have the following 
 \begin{theorem}
Let $n\geq1, s\in\mathbb{R}$, $I$ be the interval $[ 0, T]$, $0<T<\infty$,  (a,b) be Schr$\ddot{o}$dinger-admissible pair, and $ (a',b')$ be the conjugate pair of (a,b). If u is a solution to the Cauchy problem (2.1), then we have the following estimates,
\begin{multline}
\|u\|_{L^{p}_IL^r}+\|u(T,\cdot)\|_{{{H}}^s(\mathbb{R}^n)}+\|\partial_t u(T,\cdot)\|_{{{H}}^{s-2}(\mathbb{R}^n)}\\
\lesssim(1+|T|^{\frac{1}{p}+1})(\|f\|_{{H}^s}+\|g\|_{{H}^{s-2}}+\|F\| _{L^{a'}_I{W}^{s-2,b'}} ),
\end{multline}
where $ (p,r,s)$  satisfies the following condition 
$$2\leq p, r\leq \infty,\quad  \frac{2}{p}+\frac{n}{r}\geq\frac{n}{2}-s,\quad  n\geq 2, \quad (p, r, n)\neq(2,\infty,2).$$
\end{theorem}

\begin{proof}
Let $\beta (\xi)$ be a smooth cutoff function with the following properties
$$supp(\beta)\subset B_2(0),$$
$$supp(1-\beta)\subset\{|\xi|\geq 1\}.$$

Let $u_0=\beta(D)u,$  $F_0=\beta(D)F$, $u_1=(1-\beta(D))u,$ $F_1=(1-\beta(D))F$, where $u$ is the solution of  (2.1), then we have 
\begin{equation}
\partial_t^2u_0+\triangle^2u_0=F_0,
\end{equation}
\begin{equation}
\partial_t^2u_1+\triangle^2u_1=F_1.
\end{equation}
Since  $$|\xi|\geq1 \Longrightarrow |\xi|^{s-2}\approx (1+|\xi|^2)^{\frac{s-2}{2}},$$
 by Theorem 2.3  we have 
\begin{multline}
\|u_1\|_{L^{p}_IL^r}+\|u_1(T,\cdot)\|_{{{H}}^s(\mathbb{R}^n)}+\|\partial_t u_1(T,\cdot)\|_{{{H}}^{s-2}(\mathbb{R}^n)}\\
\lesssim\|f\|_{{H}^s}+\|g\|_{{H}^{s-2}}+\|F_1\| _{L^{a'}_I{W}^{s-2,b'}} .
\end{multline}
 Since, 
 $|\xi|\leq 1 \Longrightarrow (1+|\xi|^2)^{s_1}\approx(1+|\xi|^2)^{s_2}$ for any $s_1, s_2$, then for $u_0$,  Sobolev embedding gives,
\begin{equation}
\|u_0\|_{L^p_IL^r}\lesssim (1+|T|^{\frac{1}{p}})\|u_0\|_{L^{\infty}_IH^{s}}\approx (1+|T|^{\frac{1}{p}})\|u_0\|_{L_I^{\infty}L^2}.
\end{equation}
We define the energy of $u_0$ by 
\begin{equation}
E(u_0;t)=\int\frac{1}{2}|\partial_tu_0(t,x)|^2+\frac{1}{2}|\triangle u_0(t,x)|^2dx,
\end{equation}
the energy identity gives us
\begin{equation}
\partial_t E(u_0;t)=\int \partial_t u_0(t,x)F_0(t,x)dx.
\end{equation}
By Cauchy-Schwarz inequality
$$|\partial_t E^{\frac{1}{2}}(u_0;t)|\lesssim \|F_0\|_{L^2}\lesssim \|F_0\|_{H^{s-2}}.$$
By the fundamental theorem of calculus,
\begin{multline}
\|u_0(t,\cdot)\|_{L^2}\leq\|u_0(0,\cdot)\|_{L^2 }+\int_0^t\|\partial_t u_0(\tau,\cdot)\|_{L^2}d\tau
\leq\|u_0(0,\cdot)\|_{L^2 }+ \int_0^t E^{\frac{1}{2}}(u_0;\tau)d\tau\\\leq\|u_0(0,\cdot)\|_{L^2 }+  \int_0^t\|F_0(\tau,\cdot)\|_{L^2}d\tau \leq\|u_0(0,\cdot)\|_{L^2 }+ \|F_0\|_{L_I^1H^{s-2}}.
\end{multline}
Therefore we have 
\begin{equation}
\|u_0\|_{L^{\infty}_IH^{s}}+\|\partial_tu_0\|_{L^{\infty}_IH^{s-2}}\lesssim (1+|T|)( \|f\|_{H^s}+\|g\|_{H^{s-2}}+\|F_0\|_{L^1_I H^{s-2}}).
\end{equation}
Since $F_0(x)=\beta(x)^{\vee}\ast F(x)$, by Young's inequality,
\begin{equation}
\|F_0\|_{H^{s-2}}=||\beta^{\vee}\ast\langle D\rangle^{s-2} F\|_{H^{s-2}}\leq\|\beta^{\vee}\|_{L^{b_1}}\|\langle D\rangle^{s-2}F\|_{b'}\lesssim\|F\|_{W^{s-2,b'}},
\end{equation}
where, $\frac{1}{2}=\frac{1}{b_1}+\frac{1}{b'}-1$.
Combines (2.12) (2.15) (2.16) and (2.17), we have 
\begin{multline}
\|u_0\|_{L^{p}_IL^r}+\|u_0(T,\cdot)\|_{{{H}}^s(\mathbb{R}^n)}+\|\partial_t u_0(T,\cdot)\|_{{{H}}^{s-2}(\mathbb{R}^n)}\\
\lesssim(1+|T|^{\frac{1}{p}+1})(\|f\|_{{H}^s}+\|g\|_{{H}^{s-2}}+\|F\| _{L^{a'}_I{W}^{s-2,b'}} ).
\end{multline}
Therefore combines (2.11) and (2.18), we have the Strichartz estimates (2.8).

\end{proof}
In fact, the following counterexample tells us this Strichartz estimate is only valid locally. 
\begin{theorem}
For $T$ sufficiently large, we have
\begin{equation}
\sup_{g\in\mathcal{S}}\frac{\left\|\frac{\sin(t\triangle)}{\triangle}g\right\|_{L^\infty([0,T];H^s)}}{\|g\|_{H^s}}\geq c|T|.
\end{equation}
\end{theorem}
\begin{proof}
Taking $f=0$, the solution of the homogeneous beam equation will have the form $u(t,\cdot)=\frac{\sin(t\triangle)}{\triangle}g$. Let $\beta_\epsilon (\xi)$ be a smooth cutoff supported in 
$$supp(\beta_\epsilon)\subset \left\{ \xi:\frac{\epsilon^2}{2}\leq |\xi|^{2}\leq\frac{3\epsilon^2}{2}\right\}, \quad \epsilon \ll 1.$$
Set $\widehat{g(\xi)}=\beta_\epsilon (\xi)$, then
$$\|u(t,\cdot)\|^2_{H^s}=\int\left|\frac{\sin(t|\xi|^2)}{|\xi|^2}\beta_\epsilon(\xi)\langle \xi\rangle^s\right|^2d\xi.$$
Therefore, at $t=\frac{\pi}{2}\epsilon^{-2},$ $\|u(t,\cdot)\|_{H^s}\approx \epsilon^{-2}(\int|\beta_\epsilon(\xi)|^2 d\xi)^{\frac{1}{2}}.$ Also$$\|g\|_{H^{s-2}}\approx(\int|\beta_\epsilon(\xi)|^2 d\xi)^{\frac{1}{2}}.$$
Therefore,
\begin{align*}
\frac{\|u(\frac{\pi}{2}\epsilon^{-2},\cdot)\|_{H^s}}{\|g\|_{H^{s-2}}}
&\geq
\frac{\|u(\frac{\pi}{2}\epsilon^{-2},\cdot)\|_{H^s}}{\|g\|_{H^{s-2}}}\\
&\gtrsim \epsilon^{-2}\frac{(\int|\beta_\epsilon(\xi)|^2 d\xi)^{\frac{1}{2}}}{(\int|\beta_\epsilon(\xi)|^2 d\xi)^{\frac{1}{2}}}=\epsilon^{-2}\approx t.
\end{align*}
Therefore, we have (2.19)
for $T\gg1.$
    \end{proof}

\section{Well-posedness Theorems for Rough Data of the Beam Equation}
 In this section we are concerned with proving local well-posedness and global well-posedness for small data in $\dot{H}^s\times \dot{H}^{s-2}$ and local well-posedness in  $H^s\times H^{s-2}.$ 

  To prove the existence of the solution, we use Picard  iteration argument.  First we define  $F_\kappa(u)=\omega |u|^{\kappa-1}u$.  
Set  $u_{-1}\equiv0$, and  define $u_m, m=0,1,2,...,$ by\\
\begin{equation}
\left\{ \begin{array}{rcl}
(\partial_{t}^{2} +\triangle^{2})u_m=F_\kappa (u_{m-1}),\\
u_m\mid_{t=0}=f, \\
\partial _{t}u_m\mid_{t=0}=g.\end{array}\right.
\end{equation}
We need  show that there is a $0<T\leq \infty$  and a function $u$ so that 
\begin{equation}
u_m\rightarrow u\quad \text {and}\quad F_\kappa (u_m)\rightarrow F_\kappa (u), \quad \text{in}  \quad \mathcal{D}(S_T) \quad \text{with} \quad S_T=[0,T]\times {\mathbb{R}}^n.
\end{equation}

 For ``energy critical" and ``energy subcritical" exponents  $\kappa\leq\frac{n+4}{n-4}$ , when $f\in\dot{H}^s,  \quad g\in\dot{H}^{s-2}$, we have the following 

\begin{theorem}
Set $s=\frac{n}{2}-\frac{4}{\kappa-1},$ 
  if $ n>3,  \frac{8}{n}+1<\kappa\leq\frac{n+4}{n-4}$, then
 there is a $T>0$ a unique (weak) solution of (1.1) satisfying 
\begin{equation}
(u, \partial_t u)\in C([0,T]; \dot{H}^s\times \dot{H}^{s-2})\quad and\quad  u\in L^{\frac{(n+2)(\kappa-1)}{4}}(S_T).
\end{equation}
Moreover, there is $\epsilon(\kappa)>0$, so that if 
$$\|f\|_{\dot{H}^s}+\|g\|_{\dot{H}^{s-2}}<\epsilon(\kappa).$$
Then the one can take $T=\infty$. When $n=3, \kappa>5,$ we have the results above.

\end{theorem}

Because  the main step  is to show that the nonlinear mapping $u_m \rightarrow u_{m+1}$ is a contraction for the  proof of  the existence, we start with the following lemma.

\begin{lemma}
For given $ n>2, \frac{8}{n}+1< \kappa\leq\frac{n+4}{n-4}, s=\frac{n}{2}-\frac{4}{\kappa-1}$, then for $T>0$  if we set,
\begin{equation}
A_m(T)=\|u_m\|_{ L^{\frac{(n+2)(\kappa-1)}{4}}(S_T)} 
\quad and \quad  B_m(T)=\|u_m-u_{m-1}\|_{ L^{\frac{(n+2)(\kappa-1)}{4}}(S_T)},
\end{equation}
there is an $\epsilon_0 > 0$ so that $2A_0(T)\leq \epsilon_0$ and if $ m=0,1,2,...$
\begin{equation}
A_m(T)\leq 2A_0(T),\quad B_{m+1}(T)\leq\frac{1}{2}B_m(T).
\end{equation}
\end{lemma}
\begin{proof}
Suppose  that $ u$ is a weak solution of the nonlinear equation (1.1), by Theorem 2.3,  if $0\leq s\leq 2$, for every $T>0,$ we have the following Strichartz estimate 
\begin{multline}
\|u\|_{ L^{\frac{(n+2)(\kappa-1)}{4}}(S_T)}+\|u(T,\cdot)\|_{{\dot{H}}^s(\mathbb{R}^n)}+\|\partial_t u(T,\cdot)\|_{{\dot{H}}^{s-2}(\mathbb{R}^n)}\\
\lesssim\|f\|_{\dot{H}^s}+\|g\|_{\dot{H}^{s-2}}+\|F\| _{ L^{\frac{(n+2)(\kappa-1)}{4\kappa}}(S_T)}.
\end{multline}
Then if we write 
$$(\partial_t^2+\triangle^2)(u_{m+1}-u_{j+1})=V_\kappa(u_m,u_j)(u_m-u_j)$$
with
$$V_k(u,v)=\frac{F_\kappa(u)-F_\kappa(v)}{u-v},$$
then by (3.6), the  H$\text{\"{o}}$lder inequality and the fact that $V_\kappa(u_m,u_j)=O(|u_m|^{\kappa-1}+|u_j|^{\kappa-1}),$
%\begin{align*}
$$\|u_{m+1}-u_{j+1}\|_{ L^{\frac{(n+2)(\kappa-1)}{4}}(S_T)}
\leq C\|V_\kappa(u_m,u_j)(u_m-u_j)\|_{ L^{\frac{(n+2)(\kappa-1)}{4\kappa}}(S_T)}$$
$\leq C'\|V_\kappa(u_m,u_j)\|_{ L^{\frac{(n+2)}{4}}(S_T)}\|u_m-u_j\|_{ L^{\frac{(n+2)(\kappa-1)}{4}}(S_T)}$\\
%\end{align*}
$\leq C''(\|u_m\|^{\kappa-1}_{L^{\frac{(n+2)(\kappa-1)}{4}}(S_T)}
+\|u_j\|^{\kappa-1}_{L^{\frac{(n+2)(\kappa-1)}{4}}(S_T)})\|u_m-u_j\|_{L^{\frac{(n+2)(\kappa-1)}{4}}(S_T)}.$\\
Taking $j=-1,$ we have 
\begin{equation}
\|u_{m+1}-u_{0}\|_{L^{\frac{(n+2)(\kappa-1)}{4}}(S_T)}\leq C'' \|u_m\|^{\kappa}_{L^{\frac{(n+2)(\kappa-1)}{4}}(S_T)}.
\end{equation}

Thus if $\epsilon_0^{\kappa-1}$ is small enough so that $\epsilon_0^{\kappa-1}C''<\frac{1}{4}$ and if we assume that $A_m(T)\leq2A_0(T)$ then by (3.7) we get 
$$ A_{m+1}(T)\leq A_0(T)+\frac{1}{2}A_m(T),$$
by induction we get the result for Am(T). Taking $j=m-1$ gives $B_{m+1}(T)\leq\frac{1}{2}B_m(T)$.
\end{proof}

\begin{proof}[Proof of Theorem 3.1]
First of all, by (3.6) we have,
$$\|u_{0}\|_{L^{\frac{(n+2)(\kappa-1)}{4}}(S_T)}\leq C (\|f\|_{\dot H^s}+\|g\|_{\dot H^{s-2}}) \quad \text {for all} \quad T>0.$$
Therefore if the right side is less than $\frac{\epsilon_0}{2}$ for all $T$ take $T=\infty$. Otherwise the dominated convergence theorem furnishes $T$ sufficiently small such that
$$2\|u_{0}\|_{L^{\frac{(n+2)(\kappa-1)}{4}}(S_T)}=2A_0\leq \epsilon_0.$$
Since $B_0(T)=A_0(T)$, using the lemma, it follows that $u_m$ converges to a limit $u\in L^{\frac{(n+2)(\kappa-1)}{4}}(S_T)$ and hence in the sense of distributions. Since 
\begin{multline}
\|F_\kappa(u_{m+1})-F_\kappa(u_{m})\|_{L^{\frac{(n+2)(\kappa-1)}{4\kappa}}(S_T)}\\
\leq C'\|V_\kappa(u_{m+1},u_m)\|_{L^{\frac{(n+2)}{4}}(S_T)}\|u_{m+1}-u_m\|_{L^{\frac{(n+2)(\kappa-1)}{4}}(S_T)}
\end{multline}
By the lemma, we have $F_\kappa(u_m)\rightarrow F_\kappa(u)$ in $L^{\frac{(n+2)(\kappa-1)}{4\kappa}}(S_T)$.
Meanwhile, if we assume the initial  data belong to $C_0^\infty,$  by (3.5) and (3.6), $(u_m,\partial_t u_m)$ must be a Cauchy sequence in $C([0,T]; \dot H^s \times \dot H^{s-2})$ converging to $(u, v)$ for some $v$. An examination of the Duhamel formula reveals that $v=\partial_t u$, where 
$$ u(t,\cdot)=\cos(t\triangle)f +\frac{\sin(t\triangle)}{\triangle}g+\int_{0}^{t}\frac{\sin((t-s)\triangle)}{\triangle}F(u(s))ds.$$
Hence the proof of existence part of Theorem 3.1 with $\kappa\leq\frac{n+4}{n-4}$ is completed.

To prove the uniqueness, we first define $w(t,\cdot)=u_1(t,\cdot)-u_2(t,\cdot)$, where $u_1(t,\cdot),$ $u_2(t,\cdot)$ are two solutions of (1.1) satisfying (3.3), then $w(t,\cdot)$ is the solution of $(\partial_t^2+\triangle^2)w(t,\cdot)=V_\kappa(u_1(t,\cdot),u_2(t,\cdot))w(t,\cdot)$ with zero inital data, then we  consider the following equation
\begin{equation}
(\partial_t^2+\triangle^2)w(t,\cdot)=V_\kappa(u_1,u_2) w(t,\cdot),
\end{equation}
where $V_\kappa(u_1,u_2)\in L^{\frac{(n+2)}{4}}(S_T)$.
Let $T$ be the largest number such that 
$$\|V_\kappa(u_1,u_2)\|_{ L^{\frac{(n+2)}{4}}(S_T)}<\epsilon_s, \qquad \text{for}, \qquad  t\leq T.$$
Where $\epsilon_s$ is a universal constant to be determined. In particular, for some constant $C$, if $\epsilon_s\leq C^{-1}/2$, then by (3.6) and H$\text{\"{o}}$lder inequality
$$\|w\|_{ L^{\frac{(n+2)(\kappa-1)}{4}}(S_T)}\leq \frac{1}{2}\|w\|_{L^{\frac{(n+2)(\kappa-1)}{4}}(S_T)}.$$
Which implies $ w(t,\cdot)=0$, this implies uniqueness of solutions $u_1(t,\cdot)=u_2(t,\cdot)\in L^{\frac{(n+2)(\kappa-1)}{4}}(S_T).$ 
\end{proof}

 For the ``energy supercritical" range $\kappa>\frac{n+4}{n-4}$, we have two cases to discuss: $(f, g)\in(\dot{H}^s\times\dot{H}^{s-2})$ and  $(f, g)\in({H}^s\times{H}^{s-2})$.

 (1) Small  initial data  $f\in\dot{H}^s, g\in\dot{H}^{s-2},$ we have the following 
\begin{theorem}
Set $s=\frac{n}{2}-\frac{4}{\kappa-1}$ and assume $ n>4$. Suppose there exists an $l\in\mathbb{N},  l\geq1 $ with $\frac{n}{2}-\frac{4}{\kappa-1}-2\leq l\leq\kappa-1$,
then there is a $T>0$ a unique (weak) solution of the nonlinear beam equation (3.1) satisfying 
\begin{equation}
(u, \partial_t u)\in C([0,T]; {\dot{H}}^s\times {\dot{H}}^{s-2}) \quad and \quad  u\in L^{\frac{(n+2)(\kappa-1)}{4}}(S_T).
\end{equation}
Moreover, there is $\epsilon(\kappa)>0$,so that if 
$$\|f\|_{\dot{H}^s}+\|g\|_{\dot{H}^{s-2}}<\epsilon(\kappa),$$
then one can take $T=\infty$.
\end{theorem}

 To show that the nonlinear mapping $u_m \rightarrow u_{m+1}$ is a contraction for the  proof of  the existence of this theorem  requires a different argument from Lemma 3.2,  we have to use  a specific inequality which comes from Strichartz estimates as the following:
\begin{theorem}
Suppose that u is a solution of (1.1). Then, 
$$\|u\|_{L^{\frac{(n+2)(\kappa-1)}{4}}(S_T)}+\|(\sqrt{-\triangle })^{s-2}u\|_{L^{\frac{2(n+2)}{n-4}}(S_T)}+\|u(T,\cdot)\|_{{\dot{H}}^s(\mathbb{R}^n)}+\|\partial_t u(T,\cdot)\|_{{\dot{H}}^{s-2}(\mathbb{R}^n)}$$
\begin{equation}
\lesssim\|f\|_{\dot{H}^s}+\|g\|_{\dot{H}^{s-2}}+\|(\sqrt{-\triangle })^{s-2}F\| _{L^{\frac{2(n+2)}{n+4}}(S_T)},
\end{equation}
with $S_T=[0,T]\times\mathbb{R}^n$.
\end {theorem}

\begin{proof}
We assume $v=(\sqrt{-\triangle })^{s-2}u$, then 
$$(\partial_t^2+\triangle^2)v=(\sqrt{-\triangle })^{s-2}(\partial_t^2+\triangle^2)u=(\sqrt{-\triangle })^{s-2}F,$$
and 
$$v|_{t=0}=(\sqrt{-\triangle })^{s-2}f\in\dot{H}^2, \quad v_t|_{t=0}=(\sqrt{-\triangle })^{s-2}g\in L^2.$$
By (2.3) with $s=2$,
we have,
\begin{equation}
\|(\sqrt{-\triangle })^{s-2}u\|_{L^{\frac{2(n+2)}{n-4}}(S_T)}\lesssim\|f\|_{\dot{H}^s}+\|g\|_{\dot{H}^{s-2}}+\|(\sqrt{-\triangle })^{s-2}F\| _{L^{\frac{2(n+2)}{n+4}}(S_T)}.
\end{equation}
Choose $p= r= {\frac{(n+2)(\kappa-1)}{4}}$  and $a'=b'=\frac{2(n+2)}{n+4}$ for (2.2), we have
\begin{multline}
\|u\|_{L^{\frac{(n+2)(\kappa-1)}{4}}(S_T)}+\|u(T,\cdot)\|_{{\dot{H}}^s(\mathbb{R}^n)}+\|\partial_t u(T,\cdot)\|_{{\dot{H}}^{s-2}(\mathbb{R}^n)}\\
\lesssim\|f\|_{\dot{H}^s}+\|g\|_{\dot{H}^{s-2}}+\|(\sqrt{-\triangle })^{s-2}F\| _{L^{\frac{2(n+2)}{n+4}}(S_T)}.
\end{multline}
Combining with (3.12) and (3.13), we have (3.11).
\end{proof}

We first introduce fractional chain rule lemma,
\begin{lemma}
Let $F\in C^{l+1}(\mathbb{C}; \mathbb{C}), l\in \mathbb{N}$. Assume that there is $\kappa\geq l$ such that 
$$|\nabla^i F(z)|\leq |z|^{\kappa-i},\quad i=1,2,...,l.$$
If $\kappa>2, 0\leq s\leq l$,  $1<q<r<\infty$ obey the scaling condition $\frac{n}{q}=\frac{n\kappa}{r}-(\kappa-1)s$ then
\begin{equation}
\|F(f)-F(g)\|_{\dot{W}^{s,q}(\mathbb{R}^n)}\lesssim(\|f\|_{\dot{W}^{s,r}(\mathbb{R}^n)}+\|g\|_{\dot{W}^{s,r}(\mathbb{R}^n)})^{\kappa-1}\|f-g\|_{\dot{W}^{s,r}(\mathbb{R}^n)},
\end{equation}
for all $f, g \in \dot{W}^{s,r}$.
\end{lemma}
\begin{proof}
By the fundamental theorem of calculus we write$$F(f)-F(g)=\int_0^1DF((1-\theta)f+\theta g)(f-g)d\theta.$$
Let $V(f,g)=\int_0^1DF((1-\theta)f+\theta g)d\theta$, we have $F(f)-F(g)=(f-g)V(f,g)$.
By Generalized Leibniz rule (see Theorem 5, A. Gulisashvili and M.A. Kon [6])
$$\|F(f)-F(g)\|_{\dot{W}^{s,q}}\lesssim\|f-g\|_{\dot{W}^{s,r}}\|V(f,g)\|_{L^p}+\|f-g\|_{L^a}\|V(f,g)\|_{\dot{W}^{s,b}}, (\ast\ast)$$
where $\frac{1}{q}=\frac{1}{r}+\frac{1}{p}$,  $\frac{1}{q}=\frac{1}{a}+\frac{1}{b}.$
Since $$\|V(f,g)\|_{L^p}\lesssim\|f\|_{L^{p(\kappa-1)}}^{\kappa-1}+\|g\|_{L^{p(\kappa-1)}}^{\kappa-1},$$
by Sobolev embedding, we have 
$$\|V(f,g)\|_{L^p}\lesssim\|f\|_{\dot{W}^{s,r}}^{\kappa-1}+\|g\|_{\dot{W}^{s,r}}^{\kappa-1},$$
if $\frac{1}{r}-\frac{1}{p(\kappa-1)}=\frac{s}{n}.$  Combining   $\frac{1}{q}=\frac{1}{r}+\frac{1}{p}$, we have 
$$\frac{n}{q}=\frac{n\kappa}{r}-(\kappa-1)s.$$
Therefore 
\begin{equation}
\|f-g\|_{\dot{W}^{s,r}}\|V(f,g)\|_{L^p}\lesssim\|f-g\|_{\dot{W}^{s,r}}(\|f\|_{\dot{W}^{s,r}}^{\kappa-1}+\|g\|_{\dot{W}^{s,r}}^{\kappa-1}),
\end{equation}
By Liebnitz rule for fractional derivatives (see Lemma A3, T. Kato[7]), and Sobolev embedding argument similar to that in $(\ast\ast)$ above,
$$\|V(f,g)\|_{\dot{W}^{s,b}}\lesssim\int_0^1\|(1-\theta)f+\theta g\|_{\dot{W}^{s,c}}^{\kappa-1}d\theta\lesssim \|f\|_{\dot{W}^{s,c}}^{\kappa-1}+\|g\|_{\dot{W}^{s,c}}^{\kappa-1}$$
$$\|f-g\|_{L^a}\lesssim\|f-g\|_{\dot{W}^{s,r}},$$
where $\frac{1}{b}=\frac{\kappa-1}{c}-(\kappa-2)\frac{s}{n}, \quad\frac{1}{a}=\frac{1}{r}-\frac{s}{n}.$ Combining with $\frac{1}{q}=\frac{1}{a}+\frac{1}{b}$ and $\frac{n}{q}=\frac{n\kappa}{r}-(\kappa-1)s$, we have $c=r$.
Therefore we have 
\begin{equation}
\|f-g\|_{L^a}\|V(f,g)\|_{\dot{W}^{s,b}}\lesssim\|f-g\|_{\dot{W}^{s,r}}(\|f\|_{\dot{W}^{s,r}}^{\kappa-1}+\|g\|_{\dot{W}^{s,r}}^{\kappa-1})
\end{equation}
with 
$$\frac{n}{q}=\frac{n\kappa}{r}-(\kappa-1)s.$$
Combines (3.15) and (3.16) we have the result.

\end{proof}

Then we give the contraction lemma as the following:
\begin{lemma}
Set $s=\frac{n}{2}-\frac{4}{\kappa-1}$ and assume $ n>4$. Suppose there exists an $l\in\mathbb{N},  l\geq1 $ with $\frac{n}{2}-\frac{4}{\kappa-1}-2\leq l\leq\kappa-1$, then if  we set,
$$A_m(T)=\|u_m\|_{L^{\frac{(n+2)(\kappa-1)}{4}}(S_T)}+\|(\sqrt{-\triangle })^{s-2}u_m\|_{L^{\frac{2(n+2)}{n-4}}(S_T)},$$
and, 
\begin{equation}
B_m(T)=\|(\sqrt{-\triangle })^{s-2}(u_m-u_{m-1})\|_{L^{\frac{2(n+2)}{n-4}}(S_T)},
\end{equation}
there is an $\epsilon_0 > 0$ so that if $ m=0,1,2,...$
\begin{equation}
A_m(T)\leq 2A_0(T),\quad B_{m+1}(T)\leq\frac{1}{2}B_m(T),\quad \text{ if} \quad 2A_0\leq \epsilon_0.
\end{equation}
\end{lemma}

\begin{proof}
By the Liebnitz rule for fractional derivatives (see Lemma A3, T.Kato[7]) with $0\leq s-2\leq l$,
\begin{equation}
\|(\sqrt{-\triangle })^{s-2}F(u)\|_{L^q}\lesssim\|u\|^{\kappa-1}_{L^p}\|(\sqrt{-\triangle })^{s-2}u\|_{ L^r}.
\end{equation}
Where $\frac{1}{q}=\frac{\kappa-1}{p}+\frac{1}{r}.$
We apply (3.19) with $q={\frac{2(n+2)}{n+4}}$ , $p={\frac{(n+2)(\kappa-1)}{4}}$, $r= \frac{2(n+2)}{n-4}$. Specifically, this inequality along with 
(3.11) applied to the equation
$$(\partial_t^2+\triangle^2)(u_{m+1}-u_0)=F_\kappa(u_m) $$ 
gives
\begin{multline}
\|u_{m+1}\|_{L^{\frac{(n+2)(\kappa-1)}{4}}(S_T)}+\|(\sqrt{-\triangle })^{s-2}u_{m+1}\|_{L^{\frac{2(n+2)}{n-4}}(S_T)}\\
\leq C\|u_m\|^{\kappa-1}_{L^{\frac{(n+2)(\kappa-1)}{4}}(S_T)}\|(\sqrt{-\triangle })^{s-2}u_{m}\|_{L^{\frac{2(n+2)}{n-4}}(S_T)}\\
+\|u_0\|_{L^{\frac{(n+2)(\kappa-1)}{4}}(S_T)}+\|(\sqrt{-\triangle })^{s-2}u_0\|_{L^{\frac{2(n+2)}{n-4}}(S_T)}.
\end{multline}
So we have \\
\begin{align*}
A_{m+1}
&\leq C'\|u_m\|^{\kappa-1}_{L^{\frac{(n+2)(\kappa-1)}{4}}(S_T)} A_m+A_0\\
&\leq C'' A_m^{\kappa}+A_0.
\end{align*}

Then we  choose a  proper $\epsilon_0$ such that $C''2^\kappa\epsilon_0^{\kappa-1}<1,$ then we could get $A_{m+1}\leq2A_0$ by induction. 
By H$\text{\"{o}}$lder's inequality 
\begin{multline}
B_{m+1}(T)
=\|(\sqrt{-\triangle })^{s-2}(u_{m+1}-u_{m})\|_{L^{\frac{2(n+2)}{n-4}}(S_T)}\\
\leq C\|(\sqrt{-\triangle })^{s-2}(F_\kappa(u_m)-F_\kappa(u_{m-1}))\|_{L^{\frac{2(n+2)}{n+4}}(S_T)}
\end{multline}
By (3.14) with $0\leq s-2\leq l,$ we have\\
%\begin{align*}
$\|F_\kappa(u_m)-F_\kappa(u_{m-1})\|_{\dot{W}^{s-2,\frac{2(n+2)}{n+4}}(S_T)}\qquad(\star)$\\
$\leq C'(\|u_m\|_{\dot{W}^{s-2,\frac{2(n+2)}{n-4}}(S_T)}
+\|u_{m-1}\|_{\dot{W}^{s-2,\frac{2(n+2)}{(n-4)}}(S_T)})^{\kappa-1}\|u_m-u_{m-1}\|_{\dot{W}^{s-2,\frac{2(n+2)}{n-4}}(S_T)}$
$\leq C''\epsilon_0^{\kappa-1}B_m(T),$\\
%\end{align*}
leading to the desired bound if $C''\epsilon_0^{\kappa-1}<\frac{1}{2}$.
\end{proof}

With this contraction lemma, we finish the following:
\begin{proof}[Proof of Theorem 3.4]
From (3.11),
\begin{equation}
A_0(T)=\|u_0\|_{L^{\frac{(n+2)(\kappa-1)}{4}}(S_T)}+\|(\sqrt{-\triangle })^{s-2}u_0\|_{L^{\frac{2(n+2)}{n-4}}(S_T)}\\
\leq C(\|f\|_{{H}^s}+\|g\|_{H^{s-2}}).
\end{equation}
Arguing as before, we may assume $2A_0\leq \epsilon_0$.
Since  $B_0(T)\leq A_0(T)$, then  by (3.18), $u_m$ must tend to a limit in $\dot{W}^{s-2,\frac{2(n+2)}{n-4}}$. Reasoning as in the  estimating of $(\star)$, we see that $F_\kappa(u_m)$ converges to a $u$ limit  in $\dot{W}^{s-2,\frac{2(n+2)}{n+4}}$. By Fatou's lemma, 
\begin{equation}
\|u\|_{ L^{\frac{(n+2)(\kappa-1)}{4}}(S_T)}\leq \liminf_{m\rightarrow\infty}\|u_m\|_{ L^{\frac{(n+2)(\kappa-1)}{4}}(S_T)}\leq 2A_0(T)<\infty.
\end{equation}
then $u\in  L^{\frac{(n+2)(\kappa-1)}{4}}(S_T)$. By (3.17), (3.18) we have  $(\sqrt{-\triangle })^{s-2}u\in  L^{\frac{2(n+2)}{n-4}}$, and by the fractional chain rule we have $(\sqrt{-\triangle })^{s-2}F_\kappa(u)\in L^{\frac{2(n+2)}{n+4}}$. By the same way we argued in the previous case we have $(u_m,\partial_t u)\rightarrow$$(u, \partial_t u)\in C([0,T]; {\dot{H}}^s\times {\dot{H}}^{s-2})$    then the existence proof of Theorem 3.4 with  $\kappa>\frac{n+4}{n-4}$ is completed.

  To prove the uniqueness part of the theorem, we assume $u_1(t,\cdot)$ and $u_2(t,\cdot)$ are two solutions of (1.1) satisfying (3.10) then the difference $w(t,\cdot)=u_1(t,\cdot)-u_2(t,\cdot)$ satisfies the equation

$$ (\partial_t^2+\triangle^2)w(t,\cdot)=V(t,\cdot),$$
$$ w(0,x)=\partial _tw(0,x)=0,$$
 By the Strichartz estimates, we have 
$$\|w\|_{\dot{W}^{s-2,\frac{2(n+2)}{n-4}}(S_T)}\leq C \|V\|_{\dot{W}^{s-2,\frac{2(n+2)}{n+4}}(S_T)}$$
By (3.14),
$$\|(F_\kappa(u_1)-F_\kappa(u_2))\|_{\dot{W}^{s-2,\frac{2(n+2)}{n+4}}(S_T)}$$
$$\leq C'(\|u_1\|_{\dot{W}^{s-2,\frac{2(n+2)}{n-4}}(S_T)}
+\|u_2\|_{\dot{W}^{s-2,\frac{2(n+2)}{n-4}}(S_T)})^{\kappa-1}\|u_1-u_2\|_{\dot{W}^{s-2,\frac{2(n+2)}{n-4}}(S_T)}$$
Then we have,
$$\|w\|_{\dot{W}^{s-2,\frac{2(n+2)}{n-4}}(S_T)}$$$$\leq C'(\|u_1\|_{\dot{W}^{s-2,\frac{2(n+2)}{n-4}}(S_T)}+\|u_2\|_{\dot{W}^{s-2,\frac{2(n+2)}{n-4}}(S_T)})^{\kappa-1}\|w\|_{\dot{W}^{s-2,\frac{2(n+2)}{n-4}}(S_T)}$$\\
If we choose $T$ sufficiently small, $\|w\|_{\dot{W}^{s-2,\frac{2(n+2)}{n-4}}(S_T)}=0$ in $ [0,T]$, Iterating the argument it follows that $w=0$ in $ [0,T]$ for any fixed $T>0$ and this proves uniqueness.
\end{proof}
 
(2) For initial data   $f\in{H}^s, g\in{H}^{s-2},$ we expand the range of $\kappa$.
\begin{theorem}
Set $s=\frac{n}{2}-\frac{4}{\kappa-1}$ and assume $ n>4 $. Suppose there exists an $l\in\mathbb{N}$ with $\frac{n}{2}-\frac{4}{\kappa-1}-2\leq l\leq\kappa$, there is a $T>0,$ a unique ( weak) solution of the nonlinear beam equation (3.1) satisfying 
\begin{equation}
(u, \partial_t u)\in C([0,T]; {H}^s\times {H}^{s-2}) \quad and \quad  u\in L^{\frac{(n+2)(\kappa-1)}{4}}([0,T]\times \mathbb{R}^n)
\end{equation}

\end{theorem}
Similar to the proof of Theorem 3.3, we also need specific Strichartz inequality as the following for Theorem 3.8.
\begin{theorem}
Suppose that u is a solution of (1.1). Then,
$$\|u\|_{L^{\frac{(n+2)(\kappa-1)}{4}}(S_T)}+\|(\sqrt{I-\triangle })^{s-2}u\|_{ L^{\frac{2(n+2)}{n-4}}(S_T)}+\|u(T,\cdot)\|_{{{H}}^s}+\|\partial_t u(T,\cdot)\|_{H^{s-2}}$$
\begin{equation}
\lesssim\|f\|_{H^s}+\|g\|_{H^{s-2}}+\|(\sqrt{I-\triangle })^{s-2}F\| _{L^{\frac{2(n+2)}{n+4}}(S_T)},
\end{equation}
with $S_T=[0,T]\times\mathbb{R}^n, T<\infty$.
\end {theorem}
The proof of this theorem is similar to the proof of Theorem 3.4.
We also need the following:
\begin{lemma}
Given $ s=\frac{n}{2}-\frac{4}{\kappa-1}, n> 4 $. Suppose there exists an $l\in\mathbb{N}$, when $\frac{n}{2}-\frac{4}{\kappa-1}-2\leq l\leq \kappa$,  if  we set,
$$A_m(T)=\|u_m\|_{L^{\frac{(n+2)(\kappa-1)}{4}}(S_T)}+\|(\sqrt{I-\triangle })^{s-2}u_m\|_{ L^{\frac{2(n+2)}{n-4}}(S_T)},$$
and
\begin{equation}
B_m(T)=\|u_m-u_{m-1}\|_{ L^{\frac{2(n+2)}{n-4}}(S_T)},
\end{equation}
there is an $\epsilon_0 > 0$ so that if $ 2A_0(T)\leq \epsilon_0, B_0(T)\lesssim A_0(T)$ and if $ m=0,1,2,...$\\
\begin{equation}
A_m(T)\leq 2A_0(T),\quad B_{m+1}(T)\leq\frac{1}{2}B_m(T).
\end{equation}
\end{lemma}
\begin{proof} 
By the same way to prove Lemma 3.7,  using the Liebnitz rule for fractional derivatives (see Lemma A3, T.Kato[7]) with $0\leq s-2\leq l$, we easily have 
\begin{align*}
A_{m+1}
&\leq C'\|u_m\|^{\kappa-1}_{L^{\frac{(n+2)(\kappa-1)}{4}}} A_m+A_0\\
&\leq C'' A_m^{\kappa}+A_0
\end{align*}
Then we want to choose $\epsilon_0$ small such that $C''2^\kappa\epsilon_0^{\kappa-1}<1,$ for then $A_{m+1}\leq2A_0$ by induction.  
Similarly, by H$\text{\"{o}}$lder's inequality and Strichartz inequality,
\begin{align*}
B_{m+1}(T)
&=\|u_{m+1}-u_{m}\|_{L^{\frac{2(n+2)}{n-4}}}\\
&\leq C\|F_\kappa(u_m)-F_\kappa(u_{m-1})\|_{L^\frac{2(n+2)}{n+4}}\\
&\leq C'(\|u_m\|^{\kappa-1}_{L^{\frac{(n+2)(\kappa-1)}{4}}}+\|u_{m-1}\|^{\kappa-1}_{L^{\frac{(n+2)(\kappa-1)}{4}}})B_m(T)\\
&\leq C''\epsilon_0^{\kappa-1}B_m(T).
\end{align*}
If we choose a $\epsilon_0$ such that $C''\epsilon_0^{\kappa-1}<\frac{1}{2}$, we have $ B_{m+1}(T)\leq\frac{1}{2}B_m(T)$
\end{proof}
From (3.29),
\begin{equation}
A_0(T)=\|u_0\|_{L^q(S_T)}+\|(\sqrt{I-\triangle })^{s-2}u_0\|_{L^{\frac{2(n+2)}{n-4}}}\\
\leq C(\|f\|_{{H}^s}+\|g\|_{H^{s-2}}).
\end{equation}
With this contraction lemma, and since $\|u_0\| _{L^{\frac{2(n+2)}{n-4}}}\lesssim\|(\sqrt{I-\triangle })^{s-2}u_0\|_{L^{\frac{2(n+2)}{n-4}}}$, then we have $B_0(T)\lesssim A_0(T).$   By the same way to prove Theorem 3.4 (Fatou's Lemma),  we have
  $u\in L^{\frac{(n+2)(\kappa-1)}{4}}([0,T]\times \mathbb{R}^n) $. Also  if $\phi\in C_0^\infty,\langle u_m,\phi\rangle \rightarrow \langle u,\phi\rangle$ as $ m\rightarrow\infty$. Therefore,  by  H$\text{\"{o}}$lder inequality
\begin{align}
|\langle u_m,\phi \rangle|
&\leq \|(\sqrt{I-\triangle })^{s-2}u_m\|_{L^{\frac{2(n+2)}{n-4}}}\|(\sqrt{I-\triangle })^{2-s}\phi \| _{L^{\frac{2(n+2)}{n+8}}}\\
&\leq2A_0\|(\sqrt{I-\triangle })^{2-s}\phi \| _{L^{\frac{2(n+2)}{n+8}}},
\end{align}
we have 
$$|\langle u,\phi \rangle|\leq 2A_0\|(\sqrt{I-\triangle })^{s-2}\phi \| _{L^{\frac{2(n+2)}{n+8}}},$$ 
and hence $(\sqrt{I-\triangle })^{s-2}u\in L^{\frac{2(n+2)}{n-4}}$. By Strichartz estimates and the Liebnitz rule for fractional derivatives (see Lemma A3, T.Kato[7]) again, we have $$(\sqrt{I-\triangle })^{s-2}F_\kappa(u)\in {L^{\frac{2(n+2)}{n+4}}}.$$ By the same way in previous cases we proved  $(u, \partial_t u)\in C([0,T]; {{H}}^s\times {{H}}^{s-2}),$ then the existence proof of Theorem 3.8 with  $\kappa>\frac{n+4}{n-4}$ is completed. By the same way of the uniqueness proof in the previous theorem, we get the uniqueness of the solution.

\section{Scattering Theory}
In this section we consider the existence of scattering operators for the nonlinear beam equation (1.1) with initial data $f\in\dot{H}^s, g\in\dot{H}^{s-2}$.
\begin{theorem}
For $\kappa\geq1$, consider u is the solution of the equation (1.1) such that Sobolev norm of the data is sufficiently small, namely,
\begin{equation}
\|f\|_{\dot{H}^{s_c}}+\|g\|_{\dot{H}^{s_c-2}}<\epsilon.
\end{equation}
Then there exists $\epsilon >0$ small such that  for such data $(f, g)$, there is small data $( f_{+}, g_{+})\in\dot{H}^{s_c} \times\dot{H}^{s_c-2}$
 so that the solution to the free beam equation with this data,
\begin{equation}
\left\{ \begin{array}{lcl}
\partial_{t}^{2}u_{+} +\triangle^{2}u_{+}=0,\\
u_{+}\mid_{t=0}=f_{+}\in\dot{ H}^{s_c},\\
\partial _{t}u_{+}\mid_{t=0}=g_{+}\in\dot{ H}^{s_c-2}\end{array}\right.
\end{equation}
satisfies 
\begin{equation}
\lim_{T\rightarrow +\infty}\|u(T,\cdot)-u_{+}(T,\cdot)\|_{\dot{s(\kappa)}}=0,
\end{equation}
where
$$\|u(T,\cdot)\|_{\dot{s(\kappa)}}^2=\|u(T,\cdot)\|_{\dot{H}^{s_c}}^2+\|\partial_t u(T,\cdot)\|_{\dot{H}^{s_c-2}}^2.$$
Conversely, if  $( f_{-}, g_{-})\in\dot{H}^{s_c} \times\dot{H}^{s_c-2}$ has sufficiently small norm and  $ u_{-}$
 is the solution to the free beam equation with this data, then there exists a solution $u$ to (1.1) satisfying
\begin{equation}
\lim_{T\rightarrow -\infty}\|u(T,\cdot)-u_{-}(T,\cdot)\|_{\dot{s(\kappa)}}=0.
\end{equation}
Thus, the scattering operator $ S:(f_{-}, g_{-})\rightarrow ( f_{+}, g_{+})$ exists in a neighborhood of the origin in $\dot{H}^{s_c} \times\dot{H}^{s_c-2}$.
\end{theorem}

In the proof, we will only consider $\kappa \leq\frac{n+4}{n-4}$, $n>2$ case, because for $\kappa >\frac{n+4}{n-4}$ case, the method is the same, provided $l$ satisfies hypothesis of Theorem 3.4.
\begin{proof}
To prove (4.3), first we have $u\in L^{\frac{(n+2)(\kappa-1)}{4}}$  and $F_\kappa(u)\in  L^{\frac{(n+2)(\kappa-1)}{4\kappa}}$.
It follows that there is  an increasing sequence of times $ T_j,$ for which
\begin{equation}
\left(\int_{T_j}^{\infty}\int_{\mathbb{R}^n}|F_\kappa(u)|^{\frac{(n+2)(\kappa-1)}{4\kappa}}dxdt\right)^ {\frac{4\kappa}{(n+2)(\kappa-1)}}<2^{-j}
\end{equation}
Then we let $u_j$ solve the free beam equation with the same data as $u$ at $t=T_j$:
$$\left\{ \begin{array}{lcl}
\partial_{t}^{2}u_{j} +\triangle^{2}u_{j}=0,\\
u_{j}\mid_{t=T_j}=u(T_j,\cdot),\quad
\partial _{t}u_{j}\mid_{t=T_j}=\partial_tu(T_j,\cdot) \end{array}\right.$$
Then $u-u_j$ has zero data at $t=T_j$ and satisfies
 $$(\partial_{t}^{2} +\triangle^{2})(u-u_j)=F_\kappa (u).$$
Then by the Strichartz estimates (2.2)  and (4.5), we have for $T>T_j$,
\begin{equation}
\|u(T,\cdot)-u_{j}(T,\cdot)\|_{\dot{s(\kappa)}}\leq C\left(\int_{T_j}^{\infty}\int_{\mathbb{R}^n}|F_\kappa(u)|^{\frac{(n+2)(\kappa-1)}{4\kappa}}dxdt\right)^ {\frac{4\kappa}{(n+2)(\kappa-1)}}<C2^{-j},
\end{equation}
Since $u$ and $u_k$ have a the same initial data at $t=T_k$, if $k>j$ this implies 
$$\|u_k(T_k,\cdot)-u_j(T_k,\cdot)\|_{\dot{s(\kappa)}}=\|u(T_k,\cdot)-u_j(T_k,\cdot)\|_{\dot{s(\kappa)}}
\leq C2^{-j}$$
Consequently, the energy inequality yields 
$$\|u_k(0,\cdot)-u_j(0,\cdot)\|_{\dot{s(\kappa)}}\leq C2^{-j}$$
Therefore $f_j=u_j(0,\cdot), g_j=\partial_t u_j(0,\cdot)$, is a Cauchy sequence of initial data in $\dot{H}^{s_c} \times\dot{H}^{s_c-2}$.
 If we let $\lim_{j \rightarrow\infty }f_j=f_{+}, \lim_{j \rightarrow\infty} g_j =g_{+}$, then (4.6) and the energy inequality yield
$$\lim_{T\rightarrow +\infty}\|u(T,\cdot)-u_{+}(T,\cdot)\|_{\dot{s(\kappa)}}=0.$$

  To prove the second part of the theorem we define $u_{-}$ is so that the solution to the free beam equation with initial data $(f_{-},g_{-} )\in\dot{H}^{s_c} \times\dot{H}^{s_c-2}$. which has small norm.
We let $ u_{-1}=0$ and $u_0=u_{-}$ be defined by
\begin{equation}
\left\{ \begin{array}{lcl}
\partial_{t}^{2}u_{-} +\triangle^{2}u_{-}=0,\\
u_{-}\mid_{t=0}=f_{-}\in\dot{ H}^{s_c},\\
\partial _{t}u_{-}\mid_{t=0}=g_{-}\in\dot{ H}^{s_c-2},\end{array}\right.
\end{equation}
and define $u_m, m=1,2,...$ by
\begin{equation}
u_m(t,\cdot)=u_0(t,\cdot)+\int_{-\infty}^t\frac{\sin((t-s)\triangle)}{\triangle}F_{\kappa}(u_{m-1})(s,\cdot)ds,
\end{equation}
which means that $u_m$ solves  $\partial_t^2u_m+\triangle^2u_m=F_\kappa(u_{m-1}$) with initial data $(f_{-}, g_{-})$. Then use a Picard  iteration argument similar to before. Similar to Lemma 3.2, we have that $u_m$ converges to a solution $u$ of
\begin{equation}
u(t,\cdot)=u_0(t,\cdot)+\int_{-\infty}^t\frac{\sin((t-s)\triangle)}{\triangle}F_{\kappa}(u)(s,\cdot)ds,
\end{equation}
where $u\in L^{\frac{(n+2)(\kappa-1)}{4}}$, $F_\kappa(u)\in  L^{\frac{(n+2)(\kappa-1)}{4\kappa}}$  and for any $T,$ $(u, \partial_t u)\in C([0,T]; \dot{H}^{s_c}\times \dot{H}^{s_c-2}).$ 
Then we have  (4.4), therefore, the scattering operator $S: (f_{-},g_{-})\rightarrow(f_{+},g_{+})$ exists in a neighborhood of the origin in $ \dot{H}^{s_c} \times\dot{H}^{s_c-2}.$
\end{proof}

\section{ill-posedness results}
We now consider ill-posedness of the nonlinear beam equation (1.1) in the defocusing case. According to the small dispersion analysis of M. Christ, J. Colliander and T. Tao we have the result follows the following 

\begin{theorem}
Let $ n\geq 1$, $\omega=-1$ and $ \kappa>1$, if $ \kappa$ is not an odd integer, we assume $ \kappa\geq k+2$ for some integer $ k>n/2$, suppose that $ 0<s<s_c=\frac{n}{2}-\frac{4}{\kappa-1}.$ Then for any $ \epsilon>0$ there exist a real-valued solution u of  the nonlinear beam equation (1.1) and $ t\in \mathbb{R}^{+}$ such that $u(0)\in \mathcal{S}$
$$ \|u(0)\|_{H^s}<\epsilon,$$
$$u_t(0)=0,$$
and $$0<t<\epsilon,$$
$$\|u(t)\|_{H^s}>\epsilon^{-1}.$$
In particular, for  any $t>0$ the solution map $\mathcal{S}\times \mathcal{S}$ $\owns (u(0),u_t(0))\rightarrow(u(t),u_t(t))$, for Cauchy problem (1.1) fails to be continuous at $ 0$ in the $H^s\times H^{s-2}$ topology.
\end{theorem}

We analyze the small dispersion approximation for the beam equation (1.1),
\begin{equation}
 \left\{ \begin{array}{lcl}
\partial_{\tau}^{2}\phi (\tau,y)+\nu^4\triangle^{2}\phi(\tau,y)=\omega|\phi|^{\kappa-1}\phi,\\
\phi(0,y)=\phi_0(y),\\
\partial _{s}\phi(0,y)=0\end{array}\right.
\end{equation}
in the zero-dispersion limit $\nu\rightarrow 0$. 
Then for time ${t}$ define 
\begin{equation}
 u({t},x)=\phi({t},\nu x)
\end{equation}
for fixed initial datum $\phi_0$ in the small dispersion regime $\nu\rightarrow$0,  (5.1) can be transformed back into (1.1). Indeed, for any solution $\phi$ of  (5.1), by the scaling symmetry, 
\begin{equation}
\lambda^{\frac{-4}{\kappa-1}}\phi (\lambda^{-2}{t},\lambda^{-1}\nu x)
\end{equation}
also defines a solution of (1.1).

Setting $\nu=0$ in (5.1) gives the ODE
\begin{equation}
 \left\{ \begin{array}{lcl}
\partial_{s}^{2}\phi (\tau,y)= \omega|\phi|^{\kappa-1}\phi,\\
\phi(0,y)=\phi_0(y),\\
\partial _{\tau}\phi(0,y)=0\end{array}\right.
\end{equation}
 we define $\phi^0$ is this ODE solution, 

In the defocusing case  $\omega=-1$, we give the solution formula as the following
\begin{equation}
\phi^0(\tau,y)=\mathcal{C} (|\phi_0(y)|^\frac{\kappa-1}{2}\tau)\phi_0(y),
\end{equation}
where $\mathcal{C}:\mathbb{R}\rightarrow\mathbb{R}$ is the unique solution to the ODE
$$-\mathcal{C}''(\tau)=|\mathcal{C}(\tau)|^{\kappa-1}\mathcal{C}(\tau); \qquad \mathcal{C}(0)=1; \qquad \mathcal{C}'(0)=0.$$
This is the Hamiltonian flow on a two dimensional phase space with Hamiltonian
$$H:\frac{1}{2}|\mathcal{C}'(\tau)|^2+\frac{1}{\kappa+1}|\mathcal{C}(\tau)|^{\kappa+1}.$$

  It can be seen that $\mathcal{C}$ is a bounded nonconstant periodic  function and $C^{k+4}$ function for some $k>\frac{n}{2}$ since $F=\omega|\phi|^{\kappa-1}\phi\in C^{k+2}$. To  avoid causing some problems with smoothness of $|\phi_0(y)|$, we let $\phi_0(y)=(\psi(y))^{2l}$, where $\psi(y)$ is real Schwartz function and $l$ is sufficiently large.

We now use the following lemma to see that the solution of (5.1) $\phi$ may stay close to the ODE solution $\phi^0$,
when $\nu> 0$ is small.
\begin{lemma}
Let $n\geq1, \kappa\geq1, k>\frac{n}{2}$ be an integer, and if $\kappa$ is not an odd integer, then $\kappa\geq k+2$. Let  $\phi_0(y)=(\psi(y))^{2l}$, where $\psi(y)$ is a Schwartz function, and $l$ is sufficiently large, so $\phi_0$ is the square of a Schwartz function. Then there exist $ C, c,$ such that for each sufficiently small real number $0<\nu\leq c$, there exists a solution $\phi(\tau,y)$ of (5.1) for all $|\tau|\leq c|\ln\nu|^c$ such that 
\begin{equation}
||\phi(\tau)-\phi^0(\tau)||_{H^{k}}+||\phi_\tau(\tau)-\phi_\tau^0(\tau)||_{H^{k}}\leq C|\nu|, 
\end{equation}with $\phi^0 $ as in (5.5).
\end{lemma}

\begin{proof}
We define the function $F:\mathbb{C}\rightarrow \mathbb{C} $ by 
$$F(z)=|z|^{\kappa-1}z, and $$ plug in (5.4), thus 
$$\partial^2_\tau\phi^0=\omega F(\phi^0),$$
and the equation to be solved is 
$$\partial_\tau^2\phi +\nu^4 \triangle^2_y\phi =\omega F(\phi),$$
then with the ansatz 
$$\phi=\phi^0+w$$
$w $ is a solution of the Cauchy problem 
\begin{equation}
 \left\{ \begin{array}{lcl}
\partial_{\tau}^{2}w+\nu^4\triangle^{2}w=\nu^4\triangle^{2}\phi^0+\omega(F(\phi^0+w)-F(\phi^0)),\\
w(0,y)=0,\\
\partial _{\tau}w(0,y)=0\end{array}\right.
\end{equation}
Since $\kappa\geq k+2$, it's  guaranteed that $F$ is a $C^{k+2}$ function with all $k$ derivatives locally Lipschitz.
Define the $\nu-$energy of $w$ by 
\begin{equation}
E_\nu(w;\tau)=\int \frac{1}{2}|w_\tau(\tau,y)|^2+\frac{\nu^4}{2}|\triangle w(\tau,y)|^2 dy,
\end{equation}
if we have $\partial^2_\tau w+\nu^4\triangle^{2}w=\mathcal{F}$, then the energy identity gives 
\begin{equation}
\partial_\tau E_\nu (w;\tau)=\int w_\tau(\tau,y)\mathcal{F}(\tau,y)dy.
\end{equation}
By the Cauchy-Schwarz inequality
$$|\partial_\tau E_\nu^{\frac{1}{2}} (w;\tau)| \leq C\|\mathcal{F}(\tau)\|_2.$$
Similarly, if we define
$$E_{\nu,k}(w;\tau)=\sum_{j=0}^{k}\sum_{|\alpha|= j} E_\nu(\partial_y^{\alpha} w(\tau)),$$
then 
\begin{equation}
|\partial_\tau E_{\nu,k}^{\frac{1}{2}} (w;\tau)| \leq C\|\mathcal{F}(\tau)\|_{H^{k}}.
\end{equation}
and,
\begin{equation}
 E_{\nu,k}^{\frac{1}{2}} (w;\tau) \leq
\int_0^{\tau}|\partial_\tau E_{\nu,k}^{\frac{1}{2}} (w;\tau')|d\tau'\leq C\int _0^{\tau}\|\mathcal{F}(\tau')\|_{H^{k}}d\tau'.
\end{equation}

Since $\phi_0=\psi(y)^{2l},\psi(y)$ is Schwartz, $F$ is $C^{k+2}$ and $\mathcal{C}$ is $C^{k+4}$ ,
\begin{equation}
\|\phi^0\|_{H^{k}}+\|\phi^0\|_{C^k}\leq C(1+|\tau|)^{k},
\end{equation}
and 
 \begin{equation}
\|\nu^4 \triangle ^2\phi^0\|_{H^{k}}\leq C\nu^4 (1+|\tau|)^{k+4}
\end{equation}

Using Taylor formula, and the fact that $H^k$ is an algebra since $k>\frac{n}{2}$, we have
$$\|F(\phi^0+w)(\tau)-F(\phi^0)(\tau)\|_{H^{k}}\lesssim \|w(\tau)\|_{H^k}(\|w(\tau)\|_{H^k}^{\kappa-1}+\|\phi^0(\tau)\|_{H^{k}}^{\kappa-1}).$$
Define
 $$e(\tau):=\sup_{0\leq \tau'\leq \tau}E_{\nu,k}^{\frac{1}{2}} (w(\tau')),$$  which is a non-decreasing function. 
By the fundamental theorem of calculus 
\begin{equation}
\|w(\tau)\|_{H^{k}}\leq \int_0^\tau\|w_\tau(\tau')\|_{H^{k}}d\tau'\\
\leq\int_0^\tau E_{\nu,k}^{\frac{1}{2}} (w(\tau'))d\tau'\\
\leq C\tau e(\tau).
\end{equation}
  Under the assumption that $w(\tau)$ is bounded in $H^{k}$ and combining (5.13), (5.14), we have 
$$\|\nu^4 \triangle ^2\phi^0+F(\phi^0+w)(\tau)-F(\phi^0)(\tau)\|_{H^{k}}\leq C(1+|\tau|)^C (\nu^4+e(\tau)+e(\tau)^\kappa)$$
Then combines (5.12), we have the differential inequality
$$e(\tau)\leq C\int_0^\tau (1+|\tau'|)^C (\nu^4+e(\tau')+e(\tau')^\kappa)d\tau'$$
Since $e(0)=0$, by Gronwall's inequality and $w(\tau)$ is bounded in $H^{k}$, for $|\tau|\leq c|\ln \nu|^c$,  then we have $e(\tau)\leq C\nu^{\frac{7}{2}}$, and the claim follows from (5.14) if $\nu$ is sufficiently small.
\end{proof}

Now we prove Theorem 5.1:
\begin{proof}
Let  $0<\nu\ll 1$ be a parameter, we will construct solutions of (1.1) which are depending on $ \nu$, and analyze them quantitatively as $\nu \searrow 0.$ By the lemma above, for $\nu \leq c$ there exists a solution $\phi^\nu(\tau,y)=\phi(\tau,y)$ to the equation (5.1) and we have for $|\tau|\leq C|\ln \nu|^c,$
\begin{equation}
||\phi^{\nu}(\tau)-\phi^{0}(\tau)||_{H^{k}}+||\phi_\tau^{\nu}(\tau)-\phi_\tau^{0}(\tau)||_{H^{k}}\leq C|\nu|.
\end{equation}

Applying the scaling symmetry gives then solutions $u({t},x)=u^{(\nu, \lambda)}({t},x)$ to (1.1) defined by
\begin{equation}
u^{(\nu, \lambda)}({t},x) =\lambda ^{\frac{-4}{\kappa-1}}\phi^{\nu}(\lambda^{-2}{t},\lambda^{-1}\nu x).
\end{equation}
In particular, we have the initial data
\begin{equation}
u^{(\nu, \lambda)}(0,x) =\lambda ^{\frac{-4}{\kappa-1}}\phi_0(\lambda^{-1}\nu x);\qquad u_{{t}}^{(\nu, \lambda)}(0,x)=0.
\end{equation}
Assume $0<\lambda\leq \nu\ll 1$, and observe $$ [u^{(\nu, \lambda)}(0)]^{\wedge}(\xi) =\lambda ^{\frac{-4}{\kappa-1}}\left(\frac{\lambda}{\nu}\right)^n\hat{\phi_0}\left(\frac{\lambda}{\nu}\xi\right).$$
Hence $$ \|u^{(\nu, \lambda)}(0)\|_{H^s}^2= \lambda ^{\frac{-8}{\kappa-1}}\left(\frac{\lambda}{\nu}\right)^{2n}\int|\hat{\phi_0}\left(\frac{\lambda}{\nu}\xi\right)|^2(1+|\xi|^2)^sd\xi,$$
define $\eta=\frac{\lambda}{\nu}\xi,$
\begin{multline*}
 \|u^{(\nu, \lambda)}(0)\|_{H^s}^2= \lambda ^{\frac{-8}{\kappa-1}}\left(\frac{\lambda}{\nu}\right)^{n}\int|\hat{\phi_0}(\eta)|^2(1+\left|\frac{\nu}{\lambda}\eta\right|^2)^sd\eta\\
\approx \lambda ^{\frac{-8}{\kappa-1}}\left(\frac{\lambda}{\nu}\right)^{n-2s}\int_{|\eta|\geq\lambda\nu^{-1}}|\hat{\phi_0}(\eta)|^2|\eta|^{2s}d\eta+ \lambda ^{\frac{-8}{\kappa-1}}\left(\frac{\lambda}{\nu}\right)^{n}\int_{|\eta|\leq\lambda\nu^{-1}}|\hat{\phi_0}(\eta)|^2d\eta\\
=\lambda ^{\frac{-8}{\kappa-1}}\left(\frac{\lambda}{\nu}\right)^{n-2s}\int_{\mathbb{R}^n}|\hat{\phi_0}(\eta)|^2|\eta|^{2s}d\eta\\+ \lambda ^{\frac{-8}{\kappa-1}}\left(\frac{\lambda}{\nu}\right)^{n-2s}\int_{|\eta|\leq\lambda\nu^{-1}}|\hat{\phi_0}\left(\eta)|^2(\left(\frac{\lambda}{\nu}\right)^{2s}-|\eta|^{2s}\right)d\eta.
\end{multline*}
Then for some constant $C$, we have $$ \|u^{(\nu, \lambda)}(0)\|_{H^s}\leq C\lambda ^{\frac{-4}{\kappa-1}}\left(\frac{\lambda}{\nu}\right)^{\frac{n}{2}-s}=C\lambda^{s_c-s}\nu^{s-\frac{n}{2}}.$$
Given $\nu$, define $\lambda$
\begin{equation}
\lambda^{s_c-s}\nu^{s-\frac{n}{2}}=\epsilon.
\end{equation}
Consider the behavior of $u^{(\nu, \lambda)}(\tilde{t}) $ for $\tilde{t}>0$, starting with the analysis of $\phi^{0}(\tilde{t},x)$ for $\tilde{t}\gg 1$, gives,
$$\partial_x^j\phi^{0}(\tilde{t},x)=\phi_0(x){\tilde{t}}^j(\nabla_x|\phi_0(x)|^{\frac{\kappa-1}{2}})^j\mathcal{C}^{(j)}(\tilde{t}|\phi_0(x)|^{\frac{\kappa-1}{2}})+O({\tilde{t}}^{j-1}),$$
for $j=0,1,...,k.$ Since $\mathcal{C}$ and its derivatives only vanish on a countable set we thus have
$$\|\phi^{0}(\tilde{t})\|_{H^j}\sim{ \tilde{t}}^j.$$
In particular, since the Sobolev norms $H^s$ are interpolation spaces,
$$\|\phi^{0}(\tilde{t})\|_{H^s}\sim{\tilde{t}}^s,$$
whenever $s\geq 0$ is no larger than the greatest integer $\leq\kappa-1$. If $\nu\ll 1$ and $1\ll \tilde{t}\leq c|\ln\nu|^c$, (5.6) thus implies that 
\begin{equation}
\|\phi^{\nu}(\tilde{t})\|_{H^s}\sim{\tilde{t}}^s.
\end{equation}
This estimate indicates that as time progresses, the function $\phi^{\nu}(\tilde{t})$ transfers its energy to increasingly higher frequencies.
We now exploit the supercriticality of $s$ via the scaling parameter $\lambda$ to create arbitrarily large $H^s$ norms at arbitrarily small times. Applying (5.6), we have 
 $$ [u^{(\nu, \lambda)}(\lambda^2\tilde{t})]^{\wedge}(\xi) =\lambda ^{\frac{-4}{\kappa-1}}(\frac{\lambda}{\nu})^n[\phi^{\nu}(\tilde{t})]^{\wedge}(\frac{\lambda}{\nu}\xi).$$
By the change of variables $\eta: = \frac{\lambda}{\nu}\xi$
$$ \|u^{(\nu, \lambda)}(\lambda^2\tilde{t})\|_{H^s}^2\geq c \lambda ^{\frac{-8}{\kappa-1}}(\frac{\lambda}{\nu})^{n}\int|[\phi^{\nu}(\tilde{t})]^{\wedge}(\eta)|^2(1+|\frac{\nu}{\lambda}\eta|^2)^sd\eta.$$
Since $\frac{\lambda}{\nu}\leq 1$,
$$\int|[\phi^{\nu}(\tilde{t})]^{\wedge}(\eta)|^2(1+|\frac{\nu}{\lambda}\eta|^2)^sd\eta\geq (\frac{\lambda}{\nu})^{-2s}\int_{|\eta|\geq1}|[\phi^{\nu}(\tilde{t})]^{\wedge}(\eta)|^2|\eta|^{2s}d\eta$$
$$\geq(\frac{\lambda}{\nu})^{-2s}(c\|\phi^{\nu}(\tilde{t})\|_{H^s}^2-C\|\phi^{\nu}(\tilde{t})\|_{H^0}^2).$$
From (5.19) it is apparent that $\|\phi^{\nu}(\tilde{t})\|_{H^0}\ll \|\phi^{\nu}(\tilde{t})\|_{H^s}$ for $\tilde{t}\gg1.$ Thus by (5.18) and (5.19)
$$ \|u^{(\nu, \lambda)}(\lambda^2\tilde{t})\|_{H^s}\geq c \lambda ^{\frac{-4}{\kappa-1}}(\frac{\lambda}{\nu})^{\frac{n}{2}-s} \|\phi^{\nu}(\tilde{t})\|_{H^s}\geq c\epsilon {\tilde{t}}^s.$$
Therefore for $ \|u(t)\|_{H^s}$, when $\tilde{t}\approx c|\ln\nu|^c$, choose $\nu$ is small enough such that
$$c|\ln\nu|^c\gg\epsilon^{-\frac{2}{s}},$$ for $t=\lambda^2\tilde{t},$ $\nu$ sufficiently small,
$$t\approx c|\ln\nu|^c\lambda^2=C|\ln\nu|^c\nu^{2(\frac{n/2-s}{s_c-s})}\epsilon ^{\frac{2}{s_c-s}}<\epsilon,$$
we have 
$$ \|u(t)\|_{H^s}\geq \epsilon^{-1}.$$
Theorem 5.1 follows.
\end{proof}

\end{document}